\documentclass[10pt]{amsart}
\usepackage[english]{babel}
\usepackage{amssymb}
\usepackage{amsmath}
\usepackage{srcltx}
\usepackage{lscape}

\newtheorem{dummy}{Dummy}

\newtheorem{lemma}[dummy]{Lemma}
\newtheorem{theorem}[dummy]{Theorem}
\newtheorem{proposition}[dummy]{Proposition}
\newtheorem{corollary}[dummy]{Corollary}

\theoremstyle{definition}

\newtheorem{example}[dummy]{Example}

\newtheorem{remark}[dummy]{Remark}

\newcommand{\ignore}[1]{}

\author{S. Pumpl\"un}
\author{D. Thompson}

\email{susanne.pumpluen@nottingham.ac.uk; daniel.thompson1@nottingham.ac.uk}
\address{School of Mathematical Sciences\\
University of Nottingham\\ University Park\\ Nottingham NG7 2RD\\
United Kingdom }



\begin{document}

\title[The norm of a skew polynomial]
{The norm of a skew polynomial}

\maketitle

\begin{abstract}
Let $D$ be a finite-dimensional division algebra over its center and $R=D[t;\sigma,\delta]$ a skew polynomial ring. Under certain assumptions on $\delta$ and $\sigma$, the ring of central quotients $D(t;\sigma,\delta) = \{f/g \,|\, f \in D[t;\sigma,\delta], g \in C(D[t;\sigma,\delta])\}$ of $D[t;\sigma,\delta]$ is a central simple algebra with reduced norm $N$.
 We calculate the norm $N(f)$ for some skew polynomials $f\in R$ and investigate when and how the reducibility of $N(f)$ reflects the reducibility of  $f$.
\end{abstract}

%
%
\section{Introduction}

Let $A$ be a central simple algebra, $R=A[t;\sigma,\delta]$ and $A(t;\sigma,,\delta) = \{f/g \,|\, f \in A[t;\sigma,\delta], g \in C(A[t;\sigma,\delta])\}$ the ring of central quotients of $R$. Under certain assumptions on $\delta$ and $\sigma$, $A(t;\sigma,\delta)$  is a central simple algebra with norm $N$.
So far, this norm has been successfully used  when investigating factorizations of polynomials $f\in  \mathbb{F}_{q^n}[t;\sigma]$ (cf. \cite{CaB, GLN18}): $f$ is irreducible in $\mathbb{F}_{q^n}[t;\sigma]$ if and only $N(f)$ is irreducible as a polynomial in $\mathbb{F}_q[x]$, with $x=t^n$ \cite[Proposition 2.1.17]{CaB}.

The idea to compute the norm of a skew polynomial and use it to check if the polynomial is reducible or not can be employed for some families of skew polynomials in more general rings $R=D[t;\sigma,\delta]$ with $D$ a division algebra of degree $d$: we show that if $D$ has a subfield $E$ of degree $d$ and $R=D[t;\sigma]$ or $R=D[t;\delta]$ suitable, then  $N(f)\in F[x]$ is a polynomial in a suitably chosen polynomial ring $F[x]$ for all $f\in R$, and $f$ divides $N(f)$. This generalizes results by Jacobson \cite[(1.6.12), Proposition 1.7.1 (i)]{J96}, \cite[(1.6.12), p.~31]{J96}.

In particular, let $K/F$ be a finite cyclic Galois extension of degree $n$ with Galois group generated by $\sigma$. Then for all $f\in K[t;\sigma]$, we have $N(f)\in F[x]$ where $x=t^n$ and  $f$ divides $N(f)$.
If $f$ satisfies $(f,t)_r=1$, and either $f$ is not right-invariant and $n$ is prime, or  $gcd(m,n)=1$, then $f$ is irreducible in $K[t;\sigma]$ if and only $N(f)$ is irreducible as a polynomial in $F[x]$. This indeed holds more generally for all $f\in K[t;\sigma]$ such that $(f,t)_r=1$ that are not right-invariant with bound $f^*$ of degree $mn$.

We extend these results to the case that $R=D[t;\delta]$, where $\delta$ is an algebraic derivation, in particular we look at the case that  $R=K[t;\delta]$, where
 $K/F$ is a field extension in characteristic $p$ of degree $p^e$ and $\delta$ an algebraic derivation with constant field $F$. Any not right-invariant $f\in K[t;\delta]$ of degree $m$ and with bound $f^*$ of degree $mn$, is irreducible in $K[t;\delta]$ if and only $N(f)$ is irreducible as a polynomial in some polynomial ring $F[x]$.

Furthermore,  in some cases it is still true that a factorization of $N(f)$ in $F[x]$ provides a factorization of $f$.

This work is part of the second author's PhD thesis \cite{DT2020}.

%
%
\section{Preliminaries}

\subsection{Skew polynomial rings}

Let $A$ be a unital associative ring, $\sigma$ a ring endomorphism of
$A$ and $\delta:A\rightarrow A$ a \emph{left $\sigma$-derivation},
i.e. an additive map such that
$\delta(ab)=\sigma(a)\delta(b)+\delta(a)b$
for all $a,b\in S$.
The \emph{skew polynomial ring} $R=A[t;\sigma,\delta]$ is the
set of skew polynomials $g(t)=a_0+a_1t+\dots +a_nt^n$ with $a_i\in
A$, with term-wise addition and  multiplication  defined
via $ta=\sigma(a)t+\delta(a)$ for all $a\in A$ \cite{O1}.
 Define ${\rm Fix}(\sigma)=\{a\in A\,|\,
\sigma(a)=a\}$ and ${\rm Const}(\delta)=\{a\in A\,|\, \delta(a)=0\}$.
If $\delta=0$, define $A[t;\sigma]=A[t;\sigma,0]$.
If $\sigma=id$, define $A[t;\delta]=A[t;id,\delta]$.

 For $f(t)=a_0+a_1t+\dots +a_nt^n\in R$ with $a_n\not=0$, we define the {\emph degree} of $f$ as ${\rm deg}(f)=n$ and ${\rm deg}(0)=-\infty$.
 An element $f\in R$ is \emph{irreducible} in $R$ if it is not a unit and  it has no proper factors, i.e if there do not exist $g,h\in R$ with
 $1\leq {\rm deg}(g),{\rm deg} (h)<{\rm deg}(f)$ such
 that $f=gh$ \cite[p.~2 ff.]{J96}.

 Let $D$ be a division algebra. Then a polynomial  $f(t) \in D[t;\sigma,\delta]$ is  \emph{bounded} if there exists a nonzero polynomial $f^*  \in D[t;\sigma,\delta]$ such that $D[t;\sigma,\delta]f^* $ is the largest two-sided ideal of $D[t;\sigma,\delta]$ contained in $D[t;\sigma,\delta]f$. The polynomial $f^* $ is uniquely determined by $f$ up to scalar multiplication by elements of $D^\times$. $f^*$ is called the \emph{bound} of $f$.

 In this paper, $D$ will always be a central simple division algebra of degree $d$ over its center $C$.

\subsection{The minimal central left multiple of $f\in D[t;\sigma]$}

 Let $\sigma$ be an automorphism of $D$ of finite order $n$ modulo inner automorphisms, that means that $\sigma^n=i_u$ for some inner automorphism $i_u(z)=uzu^{-1}$. Then the order of $\sigma|_C$ is $n$. W.l.o.g., we choose $u\in {\rm Fix}(\sigma)$.
Let $R=D[t;\sigma]$  and define $F=C\cap {\rm Fix}(\sigma)$. Then $R$ has center
$$C(R) = F[u^{-1}t^n]=\{\sum_{i=0}^{k}a_i(u^{-1}t^n)^i\,|\, a_i\in F \}= F[x]$$
 \cite[Theorem 1.1.22]{J96}.  All polynomials $f\in R$ are bounded.

If the greatest common right divisor of $f$ and $t$ (denoted $(f,t)_r$) is one, then $f^*\in C(R)$ \cite[Lemma 2.11]{GLN18}).
  From now on we assume that $(f,t)_r=1$.
For any $f \in R=D[t;\sigma]$ with a bound in $C(R)$, we can define the \emph{minimal central left multiple of $f$ in $R$} to be the unique  polynomial of minimal degree $h \in C(R)=F[u^{-1}t^n]$ such that $h = gf$ for some $g \in R$, and such that $h(t)=\hat{h}(u^{-1}t^n)$ for some monic $\hat{h}(x) \in F[x]$.

\begin{lemma} \label{mclm exists}
 Let $f\in D[t;\sigma]$.
  \\ (i) If $(f,t)_r=1$, then the minimal central left multiple  of $f$ exists and is unique. It is equal to the bound of $f$ up to a scalar multiple from $D$.
 \\ (ii)  If $f$ is irreducible in $R$  with minimal central left multiple $h(t)=\hat{h}(u^{-1}t^n)$, then $\hat{h}(x)$ is irreducible in $F[x]$.
 \\ (iii) Let $(f,t)_r=1$ and suppose that $\hat{h}\in F[x]$ is irreducible. Then $f=f_1\cdots f_r$ for irreducible $f_i\in R$ such that $f_i\sim f_j$ for all $i,j$.
\end{lemma}

\begin{proof} (cf. \cite{AO}) \\
 (i) Let  $f^*$ be a bound of $f$. Then $f^*$ is unique up to scalar multiplication by elements in $D^\times$ and $Rf^*$ is the largest two-sided ideal of $R$ contained in the left ideal $Rf$. In particular, this implies $f^*=gf$ for some $g\in R$.
The assumption that $(f,t)_r=1$ implies that $f^*\in C(R)$ \cite[Lemma 2.11]{GLN18}) thus  $f^*$ is the unique minimal central left multiple of $f$, up to some scalar.
\\ (ii) Since $f$ is irreducible, obviously $(f,t)_r=1$. Let $h(t)=\hat{h}(u^{-1}t^n)$ be the minimal central left multiple of $f$. Suppose $\hat{h}$ is reducible in $F[x]$; that is, $\hat{h}(x)=\hat{h}_1(x)\hat{h}_2(x)$ for some $\hat{h}_i(x)\in F[x]$, such that $0<{\rm deg}(\hat{h}_i)< {\rm deg}(\hat{h})$ for $i=1, 2$. If $f$ divides $h_1(t)=\hat{h}_1(u^{-1}t^n)$ on the right, this contradicts the minimality of $h$. Moreover, as $f$ is irreducible the greatest common right divisor of $f$ and $h_1$ is 1. As $R$ is a right Euclidean domain, there exist $p,q\in R$ such that
$pf+qh_1=1.$
Multiplying everything on the right by $h_2(t)=\hat{h}_2(u^{-1}t^n)$, we obtain $pfh_2+qh=h_2$. As $f$ is a right divisor of $h$ by definition, $h=rf$ for some $r\in R$. Noting that $h_2(t)\in C(R)$, this yields
$$h_2=ph_2f+qrf=(ph_2+qr)f.$$
Thus $h_2$ is a central left multiple of $f$ of degree strictly less than $h$ contradicting the minimality of $h$. Thus  $\hat{h}(x)$ must be irreducible in $F[x]$.
 \\ (iii) If $\hat{h}\in F[x]$ is irreducible then $h$ is a two-sided maximal element in $R$, hence the irreducible factors $h_i$ of any factorization $h=h_1\cdots h_k$ of $h$ in $R$ are all similar. Now  $h(t)=p(t)f(t)$ for some $p(t)\in R$ and so comparing the irreducible factors of $f$ and $h$ and employing \cite[Theorem 1.2.9]{J96}, we see that $f=f_1\cdots f_r$ for irreducible $f_i\in R$ such that $f_i\sim f_j$ for all $i,j$ and $r\leq k$.
\end{proof}

The above results apply in particular to the special case that $D$ is a finite field extension and $\sigma\in {\rm Aut}(K)$ has order $n$. Then $R=K[t;\sigma]$  has center
$$C(R) = F[t^n]=\{\sum_{i=0}^{k}a_i(t^n)^i\,|\, a_i\in F \}= F[x]$$
where $F= {\rm Fix}(\sigma)$ \cite[Theorem 1.1.22]{J96}.

\subsection{The minimal central left multiple of $f\in D[t;\delta]$}

 Let $R=D[t;\delta]$ where $C$ is a field  of characteristic $p$ (we allow $D=C$).
Moreover, we assume that  $\delta$ is a derivation of $D$, such that $\delta|_C$ is algebraic with minimum polynomial
$$g(t)=t^{p^e}+c_1t^{p^{e-1}}+\dots+ c_et\in F[t]$$
 of degree $p^e$, where $F={\rm Const}(\delta)\cap C$.  Then $g(\delta)=id_{d_0}$ is an inner derivation of $D$.
W.l.o.g. we choose $d_0\in {\rm Const}(\delta)$, so that $\delta(d_0)=0$ \cite[Lemma 1.5.3]{J96}.
Then $R$ has center
$$C(D[t;\delta]) = F[x]=\{\sum_{i=0}^{k}a_i(g(t)-d_0)^i\,|\, a_i\in F \}$$
 with $x=g(t)-d_0$.
The two-sided $f\in D[t;\delta]$ are of the form $f(t)=uc(t)$ with $u\in D$
and $c(t)\in C(R)$ \cite[Theorem 1.1.32]{J96}. All polynomials $f\in R$ are bounded.

For every $f \in R=D[t;\sigma]$  we define the \emph{minimal central left multiple of $f$ in $R$} to be the unique  polynomial of minimal degree $h \in C(R)=F[x]$ such that $h = gf$ for some $g \in R$, and such that $h(t)=\hat{h}(g(t)-d_0)$ for some monic $\hat{h}(x) \in F[x]$:

\begin{lemma}\label{mclm exists2}
 Let $f\in R=D[t;\delta]$.
  \\ (i) The minimal central left multiple of $f$ exists and is unique. It is equal to  $f^*$ up to a scalar multiple in $D^\times$.
  \\ (ii) If $f$ is irreducible in $R$  with minimal central left multiple $h(t)=\hat{h}(g(t)-d_0)$, then $\hat{h}(x)$ is irreducible in $F[x]$.
  \\ (iii)  Let $(f,t)_r=1$ and suppose that $\hat{h}\in F[x]$ is irreducible. Then $f=f_1\cdots f_r$ for irreducible $f_i\in R$ such that $f_i\sim f_j$ for all $i,j$.
\end{lemma}

\begin{proof} (cf. \cite{AO})
\\ (i)  Let  $f^*$ be a bound of $f$. Then $f^*$ is unique up to scalar multiplication by elements in $D^\times$ and $Rf^*$ is the largest two-sided ideal of $R$ contained in the left ideal $Rf$. Since $f^*$ is two-sided, we know that $f^*(t)=dc(t)$ for some $c(t)\in C(R)$ and $d\in D^\times$. So assume w.l.o.g. that $f^*\in C(R)$. The rest of the proof is identical to the one of Lemma \ref{mclm exists}.
\\
The proofs of (ii), (iii) are identical to the one of Lemma  \ref{mclm exists} (ii), (iii).
\end{proof}

\subsection{The algebra $(A(x), \widetilde{\sigma}, ux )$}\label{sec_norm1}

Let $C/F$ be a finite cyclic field extension of degree $n$ with ${\rm Gal}(C/F) = \langle \sigma\rangle$. Let $A$ be a central simple algebra of degree $d$ with center $C$ and suppose that $\sigma$ extends to a $C$-algebra automorphism of $A$ that we also will call $\sigma$.
Then there exists an element $u \in A^\times$ such that $\sigma^n = i_u$ and $\sigma(u) = u$. These two relations determine $u$ up to multiplication with elements from $F^\times$ \cite[Lemma 19.7]{Pierce} .
Let $R=A[t;\sigma]$ and
 $$A(t;\sigma) = \{f/g \,|\, f \in A[t;\sigma], g \in C(A[t;\sigma])\}$$
be the ring of central quotients of $A[t;\sigma]$.
 Then $x =u^{-1}t^n$ is a commutative indeterminate over $A$. The centers of $A[t;\sigma]$
and $A(t;\sigma)$ are $C(A[t;\sigma]) = F[x]=F[u^{-1}t^n]$
 and
 $$C(A(t; \sigma)) = {\rm Quot}(C(A[t;\sigma])) = F(x),$$
  where ${\rm Quot}(S)$ denotes the quotient field of an integral domain $S$. Note that
 $C(A(t;\sigma))$ is a field. The ring $A(x)$ of central quotients of $A[x]$ is a subring of $A(t; \sigma)$.

$A(t; \sigma)$ is a central simple $F(x)$-algebra, more precisely,
 $$A(t; \sigma ) \cong (A(x), \widetilde{\sigma}, ux )$$
 is a cyclic generalized crossed product  \cite[Theorem 2.3]{TH}.
  Here, $\widetilde{\sigma}$ denotes the extension of $\sigma$ to $A(x)$ that fixes $x$ \cite[Lemma 2.1.]{TH}.

Note that when regarding $A(t; \sigma)$ as an $F(x)$-algebra, the choice of $u$ is lost: $x$ depends on $u$, and  different choices of $u$ lead to different actions of $F(x)$ on $A(t;\sigma)$.
Here and in the following we thus assume that $u$ is fixed and $x=u^{-1}t^n$.

 The algebra $A(t; \sigma)$ has center $F(x)$ and ${\rm deg} A(t; \sigma) =dn $.  The reduced norm $N$ of $(A(x), \widetilde{\sigma}, ux )$ is a nondegenerate form of degree $dn$ over $F(x)$.  In particular, if A is a division algebra then $A(t; \sigma)$ is a division algebra \cite[Theorem 2.2.]{TH}.

If $A$ contains a strictly maximal subfield that is Galois over $F$ with Galois group $G$, then $A(t; \sigma)$ is a crossed product with group $G$.  If $A$ is a symbol algebra over a global field $F$, then $A(t; \sigma)$ is a crossed product. If $A$ is a $p$-algebra over a global field $F$, then $A(t; \sigma)$ is a cyclic crossed product \cite[Corollary 2.4.]{TH}.

\subsection{The algebra $(D(x), \widetilde{\delta}, d_0+x )$}\label{sec_norm2}

Let $R=D[t;\delta]$ and
 $$D(t;\delta) = \{f/g \,|\, f \in D[t;\sigma], g \in C(D[t;\delta])\}$$
 the ring of central quotients of $D[t;\delta]$.
Then the center of  $D(t;\delta)$ is the field
$$C(D(t; \delta)) = {\rm Quot} (C(D[t;\delta])) = F(x)$$
 with $x=g(t)-d_0$. The ring of central quotients $D(x)$ of $D[x]$ is a subring of $D(t; \delta)$.
  Let $\widetilde{\delta}$ denote the extension of $\delta$ to $D(x)$ such that  $\widetilde{\delta}=id_t|_{D(x)}$. Then
$\widetilde{\delta}|_{C(x)}$ is algebraic with minimum polynomial $g(t)$, and
$D(t; \delta)$ is the central simple $F(x)$-algebra
 $$D(t; \delta ) \cong (D(x), \widetilde{\delta}, d_0+x ),$$
 i.e. a generalized differential algebra of degree $p^{e}d$ \cite[p. 23]{J96}. Let $N$ be its reduced norm. Then $N$ has degree $p^{e}d$. Moreover, $(D(x), \widetilde{\delta}, d_0+x )$ contains $D(x)$ as the centralizer of $C(x)$ \cite[Theorem 3.1]{Hoe}, and is free of rank $p^e$  as a left $D(x)$-module.

\subsection{The algebra $(K(x), \widetilde{\delta}, x )$}\label{sec_norm3}

Let $K$ be a field extension of characteristic $p$ and $\delta$ a derivation on $K$ that is algebraic with $F={\rm Const}(\delta)$. Let $g(t)=t^{p^e}+c_1t^{p^{e-1}}+\dots+ c_et$ be the minimum polynomial of $\delta$. Then $K/F$ is a purely inseparable extension of exponent one, and
$K^p\subset F$. More precisely, $K=F(u_1, \dots,u_e)=F(u_1)\otimes_F\dots \otimes_F F(u_e)$, $u_i^p=c_i\in F$, and $[K:F]=p^e$.
Let $R=K[t;\delta]$ then $C(R)=F[x]$ with $x=g(t)-d_0$, $d_0\in F$.
 Assume w.l.o.g. that $d_0=0$.

Let $\widetilde{\delta}$ be the extension of $\delta$ to $K(x)$ such that  $\widetilde{\delta}=id_t|_{K(x)}$, then
$\widetilde{\delta}|_{F(x)}$ is algebraic with minimum polynomial $g(t)$. The algebra
 $K(t; \delta ) \cong (K(x), \widetilde{\delta}, x )$
 is a generalized differential algebra of degree $p^{e}$ over $F(x)$ \cite[p. 23]{J96}. Let $N$ be its reduced norm. $N$ has degree $p^{e}$.

%
%

\section{Using the norm  of $(K(x),\widetilde{\sigma},x)$}


We start with the special case that $A=C$ in \ref{sec_norm1} (with some small changes in our notation):
Let $K/F$ be a cyclic field extension of degree $n$ with ${\rm Gal}(K/F)=\langle \sigma\rangle$,
 $R=K[t;\sigma]$ and $x=t^n$.
Let $N$ be the reduced norm of  the cyclic algebra  $(K(x),\widetilde{\sigma},x)$ over $F(x)$ (cf. also \cite[Proposition 1.4.6]{J96}). We have $\widetilde{\sigma}|_K=\sigma$, and $N$ is a nondegenerate form of degree $n$.

 Caruso and Le Borgne \cite{CaB} use  $N(f)$ successfully  to factorize skew polynomials $f\in \mathbb{F}_q[t,\sigma]$ over finite fields. For certain $f\in K[t,\sigma]$,
the norm $N$ of $(K(x),\widetilde{\sigma},x)$ can also be used to obtain a reducibility criterium:

 For all $f\in R$, $N(f)\in F[x]$ and  $f$ divides $N(f)$ in $R$, more precisely, $N(f)=f(t)^\sharp f(t)=f(t) f(t)^\sharp$ and $f(t)^\sharp\in R$ \cite[(1.6.12), Proposition 1.7.1 (i)]{J96}.  Moreover, we have:

\begin{theorem}\label{prop:correctedII}
Let  $f(t)=a_0+a_1t+\dots +a_mt^m$ have degree $m$. Then
 $$N(f(t))=N_{K/F}(a_0)+\dots + (-1)^{m(n-1)}N_{K/F}(a_m)x^m.$$
\end{theorem}

This is the generalized and corrected version of \cite[Proposition 1.7.1 (ii)]{J96}, which stated $(-1)^{mn} N_{K/F}(a_m)x^m$ for the leading term, and also required $m<n$.
Furthermore, our proof fixes a small mistake in the proof of \cite[Lemma 2.1.15]{CaB}.

\begin{proof}
Write $f(t)=a_0+a_1t+\dots +a_mt^m$ as
$$f(t)=P_0(x)+P_1(x)t+\dots+P_{n-1}(x)t^{n-1}$$
 with $P_i(x) \in K[x]$.
We can use verbatim the same proof as given in \cite[Lemma 2.1.15]{CaB} to obtain the matrix in $M_{n}(K[x])$ representing the left multiplication $\rho(f(t))$ with respect to the basis $1,t,\dots,t^{n-1}$:
   we have
 {\begin{equation*}
 \rho(f(t))= \left( \begin{array}{cccccccc}
  P_0                                &       t^n \sigma(P_{n-1})                                     &     \cdots  &   &   &     &     t^n\sigma^{n-1}(P_{1})\\
            P_1                       &         \sigma(P_0)     &       \cdots    &     &     & \cdots   &   \\
  \vdots                     &            &    \ddots         &       &    &         &   \vdots\\
                                            &                                                  &                                      &                \ddots                          &                                          &             & \\
  &                                               &             &                                     &
 \ddots                       &          & \\
  \vdots                     &                                            &           &                                        &                                       & \ddots  &  t^n \sigma^{n-1}(P_{n-1})\\
 P_{n-1} &            &      \cdots   & &       & \cdots   & \sigma^{n-1}(P_0)\\
 \end{array} \right).
 \end{equation*}}
 Thus $N(f(t))$, which is the determinant of this matrix, has as constant term the constant term of
$P_0(x)\sigma(P_0(x))\cdots \sigma^{n-1}(P_{0}(x))$, which is
$a_0(x)\sigma(a_0(x))\cdots \sigma^{n-1}(a_{0}(x))=N_{K/F}(a_0)$.
 There are unique integers $k,r$, $0\leq r\leq n-1$, such that we can write $m$ as $m=kn+r$. In the sum giving the determinant of this matrix,
 the term of highest degree is
$$(-1)^{m(n-1)} P_i(x)\sigma(P_i(x))\cdots \sigma^{n-i-1}(P_{i}(x))t^n \sigma^{n-i}(P_{i}(x))\cdots \sigma^{n-1}(P_{i}(x)).$$
It has degree $m=k(n-i)+(k+1)i=kn+i$ as polynomial in $x$.
(The proof of \cite[Lemma 2.1.15]{CaB} forgot to include the factor $(-1)^{m(n-1)}$ here.)
Therefore $N(f(t))$ has as highest term the highest term of this sum.
The highest term is thus given by $(-1)^{m(n-1)} a_m\sigma(a_m)\cdots $ $\sigma^{n-1}(a_m)$ $=(-1)^{m(n-1)} N_{K/F}(a_m)$.
\end{proof}

Let $f^*$ be the bound of $f\in R$. Since $N(f)\in F[x]$ is a left multiple of $f$ that lies in $C(R)$,  the bound $f^*$ of $f$ divides $N(f)$ in $R$, thus ${\rm deg}(f^*)\leq {\rm deg}(N(f))$.
If $(f,t)_r=1$ then $f^*\in C(R)$, that means $f^*$ is up to some scalar $\alpha\in K^\times$ equal to the minimal central left multiple $h$ of $f$, where as before we write $h(t)=\hat{h}(t^n)$ with $\hat{h}(x)\in F[x]$.

\begin{remark}
In \cite[Theorem 2.9, Corollary 2.12]{GLN18}, it was proved that ${\rm deg}(f^*)\leq n \cdot {\rm deg}(f)$. We recover this result as a byproduct of the above observation that $f^*$ divides $N(f)$ in $R,$ knowing that ${\rm deg}(N(f))= n \cdot {\rm deg}(f) $ in $R$ by Theorem \ref{prop:correctedII}.
If ${\rm deg}(f)=m$ and ${\rm deg}(f^*)=mn$ in $R$, it follows immediately comparing degrees that $N(f)$ is the bound of $f$.
\end{remark}

\begin{lemma}\label{le:easy}
Let $f\in R$. If $N(f)$ is irreducible in $F[x],$ then $f$ is irreducible in $R$.
\end{lemma}

\begin{proof}
Suppose towards a contradiction that $f=gp$ for $g,p\in R$ then $N(f)=N(g)N(p)$ is reducible in $F[x],$ since both $N(g)$ and $N(p)$ lie in $F[x]$, which immediately yields the assertion.
\end{proof}

 \begin{theorem} \label{thm:norm}
 Let $f(t)\in R$ be a polynomial of degree $m$ such that $(f,t)_r=1$.
  Suppose that  ${\rm deg}(\hat{h})=m$  (this is always the case if, for instance, either $n$ is prime or $gcd(m,n)=1$).
 \\ (i) If $\hat{h}$ is irreducible in $F[x]$ then $f$ is irreducible in $R$.
  \\ (ii) If $f$ is irreducible then $N(f)$ is irreducible in $F[x]$.
\end{theorem}

 \begin{proof}
    Let $h\in R$ be the minimal central left multiple of $f$ in $R$
 with ${\rm deg}(h)=nm$. We know that $N(f)$ is a two-sided multiple of $f$ in $R$, therefore the bound $f^*$ of $f$ must divide $N(f)$ in $R$. Since we assume $(f,t)_r=1,$ so we know that $f^*\in C(R)$ and therefore $f^*$ equals $h$ up to some invertible factor in $F$. Thus
  $h(t)=\hat{h}(t^n)$ must divide $N(f)$ in $R$. Write $N(f)=g (t)h(t)$ for some $g\in R$.
 By Theorem \ref{prop:correctedII}  we have ${\rm deg}(N(f))=mn$ in $R$.
Comparing degrees in $R$ we obtain ${\rm deg} N(f)={\rm deg}(g(t))+mn=mn$, which implies ${\rm deg}(g)=0$, i.e. $g(t)=a\in K^\times$.
  This implies that $N(f)=ah(t)=a \hat{h}(t^n)$.
  \\ (i) If $\hat{h}$ is irreducible in $F[x]$ then $N(f)$ is irreducible in $F[x]$ (Lemma \ref{le:easy}), thus $f$ is irreducible in $R$.
 \\ (ii) If $f$ is irreducible then $\hat{h}$ is irreducible in $F[x]$ and so again $N(f)$ is irreducible in $F[x]$.
\end{proof}

\begin{corollary}\label{cor:rough factorization}
Let $f\in R$ be a monic  polynomial of degree. Suppose that $(f,t)_r=1$ and that ${\rm deg}(\hat{h})=m$.
\\ (i) If $N(f(t))=\hat{h}_1\cdots \hat{h}_l$ such that $\hat{h}_i\in F[x]$ is irreducible,  $1\leq i\leq l$, then there exists a decomposition $f=f_1\cdots f_l$
of $f$ into irreducible factors, such that  $N(f_{i})=\hat{h}_i$ for all $i$, $1\leq i\leq l$. The degree of $f_{i}$ in $R$ equals the degree of $\hat{h}_i$ in $F[x]$ for all $i$.
\\ (ii) Assume that $N(f)$ is the product of $l$ distinct irreducible factors $\hat{h}_1\cdots \hat{h}_l$ in $ F[x]$. Then $f$ has exactly $l!$ irreducible decompositions corresponding to each possible ordering of the factors of $N(f)$.
\end{corollary}

\begin{proof}
(i) Let $N(f)=\hat{h}_1\cdots \hat{h}_l$ with $\hat{h}_i\in F[x]$ irreducible for all $i$,  $1\leq i\leq l$.

 Since $N(f)=af^*$ with $a\in F^\times$ because of our assumption that ${\rm deg}(\hat{h})=m$ and knowing that ${\rm deg}(N(f))=m$ by Theorem \ref{prop:correctedII}, we conclude that there is a ``rough factorization''
 of $f$ given by $f=g_1\cdots g_l$ where each $g_i$ has minimal central left multiple $\hat{h}_i$  \cite[Proposition 5.2]{GLN18}.
Thus $N(f)=N(g_1)\cdots N(g_l)=\hat{h}_1\cdots \hat{h}_l$ in the commutative polynomial ring $F[x]$.
$F[x]$ is a unique factorization domain and the $l$ polynomials $\hat{h}_i$ are all irreducible in $F[x]$. Hence we can conclude
that the $l$ factors $N(g_i)$ on the left-hand side must all be irreducible as well.

  Thus there is a permutation $\pi$ such that $N(g_{\pi(i)})=\hat{h}_i$ for all $i$. Since $\hat{h}_i$ is the minimal central left multiple of $g_i$ and thus divides $N(g_i)$, we can conclude that $\pi=id$.
Moreover, if $deg(g_i)=m_i$ then ${\rm deg}(N(g_{i})) =deg(\hat{h}_i)=m_i$ in $F[x]$, so the degrees of all the $\hat{h}_i$ in $F[x]$ determine the degrees of the $g_i$. Since $(f,t)_r=1$, also $(g_i,t)_r=1$ for all $i$. By Lemma \ref{le:easy}, we conclude that the $g_i$ must be irreducible for all $i$. Denote them by $f_i$ to conform with our previous notation.
\\ (ii)
Assume that $N(f)$ is the product of $l$ distinct irreducible factors.
Then $f$ is a product of $l$ irreducible
factors by (i), containing exactly one each whose reduced norm is $\hat{h}_i$ for each irreducible $\hat{h}_i$ in $F[x]$.
 Therefore $f$ has exactly $l!$ different irreducible decompositions corresponding to each possible ordering of the distinct factors $\hat{h}_i$ of $N(f)$.
\end{proof}

%
%

\section{Using the norm of the twisted function field $A(t; \sigma)$}

Let $A$ be a central simple algebra of degree $d$ with center $C$ and $R=A[t;\sigma]$, where $\sigma$
is an automorphism of $A$  of finite order $n$ modulo inner automorphisms, i.e. $\sigma^n=i_u$, such that $\sigma|_{C}\in {\rm Aut}_F(C)$  has order $n$.
 Let $N$ be the reduced norm of $(A(x), \widetilde{\sigma}, ut )$.

  \begin{proposition}\label{prop:norm all in center}
   Let  $f(t)=a_0+a_1t+\dots+a_mt^m\in C[t;\sigma] \subset R$, and  suppose $A$ has a subfield $E$ of degree $d$. Then
 $$N(f(t))=(N_{C/F}(a_0)+\dots + (-1)^{m(n-1)}N_{C/F}(a_m)x^m)^d$$
 $$=N_{E/F}(a_0)+\dots + (-1)^{dm(n-1)}N_{E/F}(a_m) x^{dm}.$$
  \end{proposition}

\begin{proof}
$\widetilde{C}=(C(x)/F(x), \widetilde{\sigma}|_{C(x)}, ut)=(C(x)/F(x), \widetilde{\sigma}|_{C(x)},t)$ is a subalgebra  of $(A(x), \widetilde{\sigma}, ut )$ of degree $n$ over $F(x)$. Note that $\sigma|_{C}\in {\rm Aut}_F(C)$
by our global assumption on $\sigma$ at the beginning of this section. In particular, this means $\widetilde{\sigma}|_{C(x)}\in
 {\rm Aut}(C(x))$. Thus
$$N(f(t))=(N_{\widetilde{C}/F(x)}(f(t)))^d$$
for all $f(t)\in C[t;\sigma] $ by \cite[Proposition. p.304]{Pierce}. This yields the assertions:
$N_{\widetilde{C}/F(x)}(f(t))=N_{C/F}(a_0)+\dots + (-1)^{m(n-1)}N_{C/F}(a_m)x^m$ by Proposition \ref{prop:correctedII}
and
$N_{C/F}(a_0)^d=N_{E/F}(a_0)$, $N_{C/F}(a_m)^d= N_{E/F}(a_m).$
\end{proof}

\begin{theorem}\label{le:Jacobsongeneralized}
Let $A$ have a subfield $E$ of degree $d$.
Let $f\in A[t;\sigma]$.
\\ (i) $N(f)\in F[x]$,
  \\ (ii) $f$ divides $N(f)$.
  \end{theorem}

  The proof works similarly as the one of \cite[Proposition 1.7.1]{J96}:

  \begin{proof}
  (i) We have $[C:F]= n$, $[A(x):F(x)]=d^2n$, and
   $[(A(x), \widetilde{\sigma}, ut ): F(x)]=d^2n^2$. The set $\{ 1,t , \dots, t^{n-1}\}$ generates $A[t; \sigma]$ over $A[x]$.
  $(A(x), \widetilde{\sigma}, ut )$ is a central simple algebra of degree $dn$ over $F(x)$ with subalgebra $A(x)$. We regard $(A(x), \widetilde{\sigma}, ut )$ as a left module over its noncommutative subalgebra $A(x)$.
    Since $C(A[t; \sigma]) = F[x] \subset A[x]$, the set $\{ 1,t , \dots, t^{n-1}\}$ also generates $A(t; \sigma)$ over $A(x)$.
  Furthermore, we have $I_t|_{K(x)} = \sigma$, where $\sigma$ denotes the extension of $\sigma$ to $K(x)$ fixing $x$, and $A(x)\subset C_{A(t;\sigma)} (K(x))$.  \cite[Lemma 1.27]{TH} therefore shows that
$\{1, t,\dots, t^{n-1} \}$ is free over $A(x)$, thus
$$A(t;\sigma) =\bigoplus_{i=0}^{n-1} A(x)t^i.$$
 Since
  $$A[t;\sigma] =\bigoplus_{i=0}^{ n-1} A[x]t^i,$$
  and  $t^n=ux$, every $f\in R \subset (A(x), \widetilde{\sigma}, ut )$ can be written as a linear combination of $1,t, \dots, t^{n-1}$ with coefficients in $A[x]$.
We therefore obtain a representation  $\rho$ of $(A(x), \widetilde{\sigma}, ut )$ by matrices in $M_n(A(x))$ by writing
$$t^i f(t)=\sum_{j=0}^{n-1} \rho_{ij}(f(t))t^j$$
for all $f\in R\subset (A(x), \widetilde{\sigma}, ut )$ and $0\leq i,j\leq n-1$. Hence the matrix $\rho(f(t))$ has entries in $A[x]$ for every $f\in R$.
Since $A$ has a subfield $E$ of degree $d$, we can regard $A$ as a left module over $E$. Let $\{v_1,\dots,v_d\}$ be a basis for $A$ over $E(x)$. Then  $\{v_1,\dots,v_d, v_1t,\dots,v_dt,\dots v_dt^{n-1}\}$ is a basis of $(A(x), \widetilde{\sigma}, ut )=A(t;\sigma)$ as a left module over $E$
and we now analogously obtain a representation $\rho$ of $(A(x), \widetilde{\sigma}, ut )$ by matrices in $M_{dn}(E(x))$ with respect to that basis.

For $f(t)\in R$, the matrix $\rho(f(t))$ has entries in $E[x]$, therefore it follows that
$$N(f(t))={\rm det} (\rho(f(t)))\in E[x]\cap F(x) =F[x].$$
(ii) Similarly as in (i), it can be shown that all the coefficients of the characteristic polynomial of $\rho(f(t))$ are contained in
  $F[x]$ (cf. also \cite[Proposition, p. 295]{Pierce}) and thus $f(t)^\sharp \in R$ by \cite[(1.6.12)]{J96}. Since $N(f(t))=f(t)f(t)^\sharp=f(t)^\sharp f(t)$, it follows that $f(t)$ divides $N(f)$ in $R$.
  \end{proof}

Let $f\in R=A[t;\sigma]$ have degree $m$ and bound $f^*$.
 Observe that since $N(f)\in F[x]=C(R)$ is a left multiple of $f$ by Lemma \ref{le:Jacobsongeneralized}, $f^*$ divides $N(f)$ in $R,$ so that ${\rm deg}(f^*)\leq {\rm deg}(N(f))$, as also shown in \cite[Theorem 2.9, Corollary 2.12]{GLN18}.

The next result further generalizes \cite[Proposition 1.7.1]{J96}, i.e. Theorem \ref{prop:correctedII}.

\begin{theorem}\label{thm:norm6}
 Let $A=D$ be a division algebra which has a subfield $E$ of degree $d$. Then for any $f\in D[t;\sigma]$ of degree $m$, $N(f)$ has degree $dm$.
\end{theorem}

\begin{proof}
 Write $m=kn+r$ for some $0\leq r<n$.
 Substituting $t^{n}=ux$, we obtain $f(t)=P_0(x)+P_1(x)t+\dots+ P_{n-1}(x)t^{n-1}\in D[x][t;\sigma]$ where
 \begin{align*}
 P_i(x)=\begin{cases}
  a_i+\dots +a_{i+kn}(ux)^k &\text{\quad for $i\leq r$,}\\
  a_i+\dots +a_{i+(k-1)n}(ux)^{k-1} &\text{\quad for $i>r$.}
  \end{cases}
 \end{align*}
 Computing the left regular representation of $\rho:D[t;\sigma]\to M_n(D(x))$, we have
\begin{equation*}
 \rho(f(t))= \left( \begin{array}{ccc}
 Q_{1,1}(x) & \cdots & Q_{1,n}(x)\\
  \vdots & & \vdots\\
 Q_{n,1}(x) & \cdots & Q_{n,n}(x)\\
 \end{array} \right)
 \end{equation*}
for some $Q_{i,j}(x)\in D[x]$ satisfying  $$t^{i-1}f=\sum_{j=1}^{n}Q_{ij}(x)t^{j-1}, \qquad 1\leq i\leq n,$$ \cite[Proposition 1.6.9]{J96}. Moreover, it follows that
\begin{align*}{\rm deg}(Q_{i,j})=\begin{cases}
 {\rm deg}(P_{j-i}) &\text{\quad for $i\leq j$,}\\
 {\rm deg}(P_{n+j-i})+1 &\text{\quad for $i>j$.}
 \end{cases}
 \end{align*}
Comparing the above equation with the expressions for $P_i(x)$, it follows that
\begin{align*}{\rm deg}(Q_{i,j})\leq\begin{cases}
 k-1 &\text{\quad for $i\leq j$ and $j-i>r$,}\\
 k &\text{\quad for $i\leq j\leq r+i$ or $j<i<n-r+j$,}\\
 k+1 &\text{\quad for $i>j$ and $i-j\geq n-r$.}
 \end{cases}
 \end{align*}
with $Q_{i,j}(x)=\sigma^{i-1}(a_m)u^kx^k+\dots$ for $j-i=r$ and $Q_{i,j}(x)=\sigma^{i-1}(a_m)u^{k+1}x^{k+1}+\dots$ for $i-j=n-r$.

 This means the bottom left $r\times r$ minor of $\rho(f(t))$ has elements of degree at most $k+1$ in lower triangular entries (including the diagonal which attains this maximum degree) and the top right $n-r\times n-r$ minor of $\rho(f(t))$ has elements of degree at most $k-1$ in the upper triangular entries (excluding the diagonal which has elements of exactly degree $k$). Every other element of $\rho(f(t))$ has degree at most $k$.

As $D$ has a subfield of degree $d$, there exists a left regular representation $\omega:D\to M_d(E)$ which extends to $D[x]$ by setting $\omega(x)=xId$. The $d\times d$ block matrices representing $Q_{i,j}(x)$ are inserted for every entry $Q_{i,j}(x)$ in $\rho(f(t))$ to obtain a representation for $D[t;\sigma]$ in $M_{dn}(E[x])$.

As $\omega$ is additive and $\omega(x)$ is a diagonal matrix, $\omega(\sigma^j(g(x)))=\omega(\sigma^j(g_k)u^k)(xId)^k+\dots+\omega(\sigma^j(g_0))$ for any polynomial $g(x)=\sum_{i=0}^k g_ix^i\in D[x]$. As we are computing the determinant only to find the degree of $N(f(t))$, it is sufficient to only consider the term of highest degree in $Q_{i,j}(x)$ and ignore all terms of lower degree. We truncate $Q_{i,j}(x)$ at the highest term and apply $\omega$ to all the entries of the matrix. To determine the term of highest degree, we expand the determinant along the columns and consider only the portion of the determinant expansion which yields the maximum possible degree.
 As  $\omega(a_mu^k)=\omega(a_m)\omega(u)^k$ is invertible, there are no zero columns in $\omega(a_mu^k)$ so it is always possible to find an expansion of the matrix yielding the highest degree. Hence the degree of $N(f(t))$ is at most $dr(k+1)+d(n-r)k=d(kn+r)=dm.$
We wish to show the coefficient of $x^{dm}$ in $N(f(t))$ is non-zero: following a suitable expansion of $\omega\circ\rho(f(t))$
the coefficient of $x^{dm}$ is equal to
$$\pm{\rm det}(\omega(a_mu^k)){\rm det}(\omega(\sigma(a_mu^k))\cdots
{\rm det}(\omega(\sigma^{n-r-1}(a_mu^k)){\rm det}(\omega(\sigma^{n-r}(a_mu^{k+1})) \cdots$$
$$
{\rm det}(\omega(\sigma^{n-1}(a_mu^{k+1})).$$
As $\sigma$ is an automorphism,  ${\rm det}(\omega(\sigma^i(a_mu^k))\neq 0$ by our assumption on $\omega(a_mu^k)$. Thus the coefficient of $x^{dm}$ is non-zero and ${\rm deg}(N(f(t))=dm$.
\end{proof}

\begin{example}
We show the details of the above calculations for $d=2$, $n=3$ and $m=7$;  an actual computation of the matrix becomes difficult
for any $d,n,m$.  For $f(t)=a_0+a_1t+\dots+a_7t^7\in D[t;\sigma]$, where we assume $D$ has a subfield $E$ of degree $d$, and $t^3=ux$, we obtain
{\begin{equation*}
 \rho(f(t))= \left( \begin{array}{ccc}
  a_0+a_3ux+a_6u^2x^2    				&   a_1+a_4ux+a_7u^2x^2     		&   a_2+a_5ux\\
  \sigma(a_2ux+a_5u^2x^2)  	 			&   \sigma(a_0+a_3ux+a_6u^2x^2) 	&   \sigma(a_1+a_4ux+a_7u^2x^2)\\
 \sigma^2(a_1ux+a_4u^2x^2+a_7u^3x^3)    &    \sigma^2(a_2ux+a_5u^2x^2)         &   \sigma^2(a_0+a_3ux+a_6u^2x^2)\\
 \end{array} \right).
 \end{equation*}}
 Truncate the polynomials in the matrix at their highest terms and apply $\omega:D\to M_2(E)$ to get
 \begin{align*}
 	&\left( \begin{array}{ccc}
  		\omega(a_6u^2x^2)    				&   \omega(a_7u^2x^2)     		&   \omega(a_5ux)\\
  		\omega(\sigma(a_5u^2)x^2 ) 	 			&   \omega(\sigma(a_6u^2)x^2) 	&   \omega(\sigma(a_7u^2)x^2)\\
		 \omega(\sigma^2(a_7u^3)x^3)    &    \omega(\sigma^2(a_5u^2)x^2)        &   \omega(\sigma^2(a_6u^2)x^2)\\
 	\end{array} \right)\\
 =& \left( \begin{array}{cccccc}
 		a^{(1,1)}_{1,1}x^2 & a^{(1,2)}_{1,1}x^2 & a^{(1,1)}_{1,2}x^2 & a^{(1,2)}_{1,2}x^2 & a^{(1,1)}_{1,3}x &  a^{(1,2)}_{1,3}x\\
 		a^{(2,1)}_{1,1}x^2 & a^{(2,2)}_{1,1}x^2 & a^{(2,1)}_{1,2}x^2 & a^{(2,2)}_{1,2}x^2 & a^{(2,1)}_{1,3}x &  a^{(2,2)}_{1,3}x\\
 		a^{(1,1)}_{2,1}x^2 & a^{(1,2)}_{2,1}x^2 & a^{(1,1)}_{2,2}x^2 & a^{(1,2)}_{2,2}x^2 & a^{(1,1)}_{2,3}x^2 &  a^{(1,2)}_{2,3}x^2\\
 		a^{(2,1)}_{2,1}x^2 & a^{(2,2)}_{2,1}x^2 & a^{(2,1)}_{2,2}x^2 & a^{(2,2)}_{2,2}x^2 & a^{(2,1)}_{2,3}x^2 &  a^{(2,2)}_{2,3}x^2\\
 		a^{(1,1)}_{3,1}x^3 & a^{(1,2)}_{3,1}x^3 & a^{(1,1)}_{3,2}x^2 & a^{(1,2)}_{3,2}x^2 & a^{(1,1)}_{1,3}x^2 &  a^{(1,2)}_{3,3}x^2\\
 		a^{(2,1)}_{3,1}x^3 & a^{(2,2)}_{3,1}x^3 & a^{(2,1)}_{3,2}x^2 & a^{(2,2)}_{3,2}x^2 & a^{(2,1)}_{1,3}x^2 &  a^{(2,2)}_{3,3}x^2\\
 	\end{array} \right),
 \end{align*}
for some $a^{(i,j)}_{k,l}\in E$ for $i,j\in\{1,2\}$ and $k,l\in\{1,2,3\}$. Then the determinant of the above matrix is equal to\\
\hspace*{-2.5cm}\vbox{\begin{align*}
a^{(1,1)}_{3,1}x^3  &\left| \begin{array}{ccccc}
 		 a^{(1,2)}_{1,1}x^2 & a^{(1,1)}_{1,2}x^2 & a^{(1,2)}_{1,2}x^2 & a^{(1,1)}_{1,3}x &  a^{(1,2)}_{1,3}x\\
 		 a^{(2,2)}_{1,1}x^2 & a^{(2,1)}_{1,2}x^2 & a^{(2,2)}_{1,2}x^2 & a^{(2,1)}_{1,3}x &  a^{(2,2)}_{1,3}x\\
 		 a^{(1,2)}_{2,1}x^2 & a^{(1,1)}_{2,2}x^2 & a^{(1,2)}_{2,2}x^2 & a^{(1,1)}_{2,3}x^2 &  a^{(1,2)}_{2,3}x^2\\
 		 a^{(2,2)}_{2,1}x^2 & a^{(2,1)}_{2,2}x^2 & a^{(2,2)}_{2,2}x^2 & a^{(2,1)}_{2,3}x^2 &  a^{(2,2)}_{2,3}x^2\\
 		 a^{(2,2)}_{3,1}x^3 & a^{(2,1)}_{3,2}x^2 & a^{(2,2)}_{3,2}x^2 & a^{(2,1)}_{1,3}x^2 &  a^{(2,2)}_{3,3}x^2\\
 	\end{array} \right|
 -a^{(2,1)}_{3,1}x^3 \left| \begin{array}{ccccc}
 		 a^{(1,2)}_{1,1}x^2 & a^{(1,1)}_{1,2}x^2 & a^{(1,2)}_{1,2}x^2 & a^{(1,1)}_{1,3}x &  a^{(1,2)}_{1,3}x\\
 		 a^{(2,2)}_{1,1}x^2 & a^{(2,1)}_{1,2}x^2 & a^{(2,2)}_{1,2}x^2 & a^{(2,1)}_{1,3}x &  a^{(2,2)}_{1,3}x\\
 		 a^{(1,2)}_{2,1}x^2 & a^{(1,1)}_{2,2}x^2 & a^{(1,2)}_{2,2}x^2 & a^{(1,1)}_{2,3}x^2 &  a^{(1,2)}_{2,3}x^2\\
 		 a^{(2,2)}_{2,1}x^2 & a^{(2,1)}_{2,2}x^2 & a^{(2,2)}_{2,2}x^2 & a^{(2,1)}_{2,3}x^2 &  a^{(2,2)}_{2,3}x^2\\
 		 a^{(1,2)}_{3,1}x^3 & a^{(1,1)}_{3,2}x^2 & a^{(1,2)}_{3,2}x^2 & a^{(1,1)}_{1,3}x^2 &  a^{(1,2)}_{3,3}x^2\\
 	\end{array} \right|+ \dots\\
& = a^{(1,1)}_{3,1}a^{(2,2)}_{3,1}x^6  \left| \begin{array}{cccc}
 		  a^{(1,1)}_{1,2}x^2 & a^{(1,2)}_{1,2}x^2 & a^{(1,1)}_{1,3}x &  a^{(1,2)}_{1,3}x\\
 		 a^{(2,1)}_{1,2}x^2 & a^{(2,2)}_{1,2}x^2 & a^{(2,1)}_{1,3}x &  a^{(2,2)}_{1,3}x\\
 		  a^{(1,1)}_{2,2}x^2 & a^{(1,2)}_{2,2}x^2 & a^{(1,1)}_{2,3}x^2 &  a^{(1,2)}_{2,3}x^2\\
 		 a^{(2,1)}_{2,2}x^2 & a^{(2,2)}_{2,2}x^2 & a^{(2,1)}_{2,3}x^2 &  a^{(2,2)}_{2,3}x^2\\
 	\end{array} \right|
 -a^{(2,1)}_{3,1}a^{(1,2)}_{3,1}x^6 \left| \begin{array}{cccc}
 		 a^{(1,1)}_{1,2}x^2 & a^{(1,2)}_{1,2}x^2 & a^{(1,1)}_{1,3}x &  a^{(1,2)}_{1,3}x\\
 		  a^{(2,1)}_{1,2}x^2 & a^{(2,2)}_{1,2}x^2 & a^{(2,1)}_{1,3}x &  a^{(2,2)}_{1,3}x\\
 		  a^{(1,1)}_{2,2}x^2 & a^{(1,2)}_{2,2}x^2 & a^{(1,1)}_{2,3}x^2 &  a^{(1,2)}_{2,3}x^2\\
 		  a^{(2,1)}_{2,2}x^2 & a^{(2,2)}_{2,2}x^2 & a^{(2,1)}_{2,3}x^2 &  a^{(2,2)}_{2,3}x^2\\
 	\end{array} \right|+ \dots\\
& = (a^{(1,1)}_{3,1}a^{(2,2)}_{3,1}-a^{(2,1)}_{3,1}a^{(1,2)}_{3,1})x^6 \left| \begin{array}{cccc}
 		 a^{(1,1)}_{1,2}x^2 & a^{(1,2)}_{1,2}x^2 & a^{(1,1)}_{1,3}x &  a^{(1,2)}_{1,3}x\\
 		  a^{(2,1)}_{1,2}x^2 & a^{(2,2)}_{1,2}x^2 & a^{(2,1)}_{1,3}x &  a^{(2,2)}_{1,3}x\\
 		  a^{(1,1)}_{2,2}x^2 & a^{(1,2)}_{2,2}x^2 & a^{(1,1)}_{2,3}x^2 &  a^{(1,2)}_{2,3}x^2\\
 		  a^{(2,1)}_{2,2}x^2 & a^{(2,2)}_{2,2}x^2 & a^{(2,1)}_{2,3}x^2 &  a^{(2,2)}_{2,3}x^2\\
 	\end{array} \right|+ \dots.
\end{align*}}
Comparing this to the matrix in terms of $\omega$,  this is equal to
\begin{equation*}
{\rm det}(\omega(\sigma^2(a_7u^3))x^6\left| \begin{array}{ccc}
  		\omega(a_7u^2x^2)     		&   \omega(a_5ux)\\
  	   \omega(\sigma(a_6u^2)x^2) 	&   \omega(\sigma(a_7u^2)x^2)\\
 	\end{array} \right|+\dots.
\end{equation*}
Repeating this with the remaining block matrices, this yields
$$N(f(t))={\rm det}(\omega(\sigma^2(a_7u^3)){\rm det}(\omega(\sigma(a_7u^2)){\rm det}(\omega(a_7u^2))x^{14}+ \text{ terms of lower degree.}$$
\end{example}

From now on, we assume  that
$$A=(E/C,\gamma,a) \text{ is a cyclic algebra over } C \text{ of degree } d,$$
$$\sigma|_E\in {\rm Aut}(E) \text{ such that }\gamma\circ\sigma=\sigma\circ\gamma \text{ and } u\in E.$$
 Then $\sigma|_E$ has order $n$. Write $m=kn+r$ for some $0\leq r<n$.

\begin{theorem} \label{thm:norm3}
  For $f(t)=a_0+a_1t+\dots+a_mt^m\in E[t;\sigma]\subset A[t;\sigma]$, we have
  $$N(f(t))=N_{E/F}(a_0)+\dots +(-1)^{dr(n-1)}N_{E/F}(a_m)N_{E/C}(u)^{r}x^{dm}.$$
  \end{theorem}

 \begin{proof}
The proof is similar to the proof of Theorem \ref{thm:norm6}. Using the same notation as in that proof, the entries $Q_{i,j}(x)\in D[x]$ of $\rho(f(t))$ are determined by the relation $t^{i-1}f=\sum_{j=1}^{n}Q_{ij}(x)t^{j-1}, \: 1\leq i\leq n,$ and
 {\begin{equation*}
 \rho(f(t))= \left( \begin{array}{cccc}
  P_0(x)     &   P_1(x) & \cdots     &     P_{n-1}(x)\\
  \sigma(P_{n-1}(x))ux   &   \sigma(P_0(x))  & \cdots   &   \sigma(P_{n-2}(x))\\
  \vdots                     &      \ddots        &          &   \vdots\\
 \sigma^{n-m}(P_{r}(x))ux &     &     \ddots       & \sigma^{n-m}(P_{r-1}(x))\\
  \vdots                     &                              & \ddots  &  \vdots\\
 \sigma^{n-1}(P_1(x))ux &       \sigma^{n-1}(P_2(x))ux  & \cdots   & \sigma^{n-1}(P_0(x))\\
 \end{array} \right),
 \end{equation*}}
where $P_i(x)\in E[x]$ for all $i$.  Let $\{v_1,\dots,v_d\}$ be a canonical basis for $A$ as a left $E$-module. Then
  $\{v_1,\dots,v_d, v_1t,\dots,v_dt,\dots v_dt^{n-1}\}$
   is a basis of $(A(x), \widetilde{\sigma}, ut )$ as a left module over $E(x)$
and we now analogously obtain a representation $\rho$ of $(A(x), \widetilde{\sigma}, ut )$ by matrices in $M_{dn}(E(x))$ with respect to that basis. This representation is given by an $nd\times nd$ matrix obtained as follows:

 Let $\omega$ be the representation of $A$ in $M_d(E)$ which is extended to a representation of $A[x]$ in $M_d(E[x])$ by setting $\omega(x)=xI_d$.
The $d\times d$ block matrices representing the entries of $\rho(f(t))$ are inserted for every entry of the previous $n\times n$ matrix (cf. for instance \cite[p. 298]{Pierce}) with $\sigma$ extended to $M_d(E)$ by acting entry-wise.
For all $a\in E$, the matrix $\omega(a)\in M_d(E)$ is a $d\times d$ diagonal matrix given by \begin{equation*}
\left( \begin{array}{cccc}
a & 0 & \dots & 0 \\
0 & \gamma(a) &            & 0 \\
\vdots & & \ddots & \\
0 & 0 & \dots & \gamma^{d-1}(a) \\
\end{array} \right).
\end{equation*}
As a consequence, we note that $\omega(a_ix)=\omega(a_i)\omega(x)=\omega(a_i)(xI_d)$ and $\omega(a_ia_j)=\omega(a_i)\omega(a_j)$ for all $a_i,a_j\in E$.
Thus we have
{\begin{equation*}
 \omega\circ\rho(f(t))= \left( \begin{array}{cccc}
  \omega(P_0(x))     &   \omega(P_1(x)) & \cdots     &     \omega(P_{n-1}(x))\\
  \omega(\sigma(P_{n-1})(x))\omega(u)xId   &   \omega(\sigma(P_0(x)))  & \cdots   &   \omega(P_{n-2}(x))\\
  \vdots                     &      \ddots        &          &   \vdots\\
 \omega(\sigma^{n-m}(P_{r}(x)))\omega(u)xId &     &     \ddots       & \omega(\sigma^{n-m}(P_{r-1}(x)))\\
  \vdots                     &                              & \ddots  &  \vdots\\
 \omega(\sigma^{n-1}(P_1(x)))\omega(u)xId &       \omega(\sigma^{n-1}(P_2(x)))\omega(u)xId  & \cdots   & \omega(\sigma^{n-1}(P_0(x)))\\
 \end{array} \right)
 \end{equation*}}
 where $\omega\circ\rho(f(t))$ is a $dn\times dn$ matrix in $M_{dn}(E[x])$.

As the $\omega(\sigma^j(P_i(x)))$ are pairwise commutative matrices, we may calculate the determinant of $\omega\circ\rho(f(t))$ by first evaluating the $n\times n$ determinant with entries in $M_d(E)$, then evaluating the resulting $d\times d$ matrix which has entries in $E$ \cite[Lemma 1, p.~546]{Bourbaki}. Thus we obtain ${\rm{det}}(\omega\circ\rho(f(t)))={\rm det}(H),$ where \begin{align*}
H=&\omega(P_0(x))\sigma(\omega(P_0(x)))\dots\sigma^{n-1}(\omega(P_0(x)))+\dots\\
+ &(-1)^{r(n-r)}\omega(P_{r}(x))\sigma(\omega(P_{r}(x)))\dots\sigma^{n}(\omega(P_{r}(x)))\omega(u)^{r}(xI_d)^{r}.
\end{align*}
 As each $\omega(P_i(x))$ is a diagonal matrix in $M_d(E)$ for all $0\leq i\leq n-1$, $H$ is the diagonal matrix in $M_d(E)$ given by the entries
\begin{align*}
H_{ii}=&\gamma^{i-1}[P_0(x)\sigma(P_0(x))\dots \sigma^{n-1}(P_0(x))+\dots\\
+&(-1)^{r(n-1)}P_{r}(x)\sigma(P_{r}(x))\dots\sigma^{n-1}(P_{r}(x))u^{r}]x^{r}.
\end{align*}
Hence
\begin{align*}{\rm det}(H)= &\prod_{i=1}^d \gamma^{i-1}[P_0(x)\sigma(P_0(x))\dots \sigma^{n-1}(P_0(x))+\dots\\
+&(-1)^{m(n-1)}P_{r}(x)\sigma(P_{r}(x))\dots\sigma^{n-1}(P_{r}(x))u^{r}x^{r})].
\end{align*}
We obtain the constant term of $N(f(t))$ by substituting $x=0$. Thus the constant term equals
$$\prod_{i=1}^d \gamma^{i-1}(a_0\sigma(a_0)\dots \sigma^{n-1}(a_0))=\prod_{i=1}^n \sigma^{i-1}(a_0\gamma(a_0)\dots \gamma^{d-1}(a_0)),$$
 since $\gamma$ commutes with $\sigma$.
 As $a_0\in E$, this is equal to $\prod_{i=1}^n \sigma^{i-1}(N_{E/C}(a_0))=N_{C/F}(N_{E/C}(a_0))=N_{E/F}(a_0).$
Similarly, the leading term of $N(f(t))$ is given by the leading term of
$$\prod_{i=1}^d \gamma^{i-1}[(-1)^{r(n-1)}P_{r}(x)\sigma(P_{r}(x))\dots\sigma^{n-1}(P_{r}(x))u^{r}x^{r}],$$
which is given by
\begin{align*}
 &\prod_{i=1}^d \gamma^{i-1}[(-1)^{r(n-1)}a_m\sigma(a_m)\dots\sigma^{n-1}(a_m)u^{r}x^{kn}x^{r})]\\
=& (-1)^{dm_0(n-1)}\left[\prod_{i=1}^d \gamma^{i-1}(a_m\sigma(a_m)\dots\sigma^{n-1}(a_m))\right] N_{E/C}(u)^{m_0}x^{d(kn+m_0)}
\end{align*}
since $u\in E$. As $\sigma$ and $\gamma$ commute and $a_m\in E$, we can express this as
\begin{align*}
  & (-1)^{dm_0(n-1)}\left[\prod_{i=1}^n \sigma^{i-1}(a_m\gamma(a_m)\dots\gamma^{d-1}(a_m))\right] N_{E/C}(u)^{m_0}x^{d(kn+m_0)}\\
 =& (-1)^{dr(n-1)}\left[\prod_{i=1}^n \sigma^{i-1}(N_{E/C}(a_m))\right] N_{E/C}(u)^{m_0}x^{d(kn+m_0)}\\
 =& (-1)^{dm_0(n-1)} N_{C/F}(N_{E/C}(a_m))N_{E/C}(u)^{r}x^{d(kn+r)}.
\end{align*}
As $d(kn+r)=dm$, this implies the assertion.
\end{proof}

 If $(f,t)_r=1$ then $f^*\in C(R)$, that means $f^*$ is up to some scalar $\alpha\in D^\times$ equal to the minimal central left multiple $h$ of $f$, where $h(t)=\hat{h}(u^{-1}t^n)$ with $\hat{h}(x)\in F[x]$. Analogously as
  Lemma \ref{le:easy} we obtain:

\begin{lemma}\label{le:easy2}
Let $f\in R=A[t;\sigma]$, where $A$ has a subfield $E$ of degree $d$.
 If $N(f)$ is irreducible in $F[x],$ then $f$ is irreducible in $R$.
\end{lemma}

 \begin{theorem} \label{thm:norm4}
 Let $A=(E/C,\gamma,a)$ be a cyclic algebra over $C$ of degree $d$ such that $\sigma|_E\in {\rm Aut}(E)$, $u\in E$, and  $\gamma\circ\sigma=\sigma\circ\gamma$.
Let $f\in R$ be a polynomial such that $(f,t)_r=1$.
  Suppose that  ${\rm deg}(\hat{h})=dm$.
 \\ (i) If $\hat{h}$ is irreducible in $F[x]$ then $f$ is irreducible in $R$.
 \\ (ii) If $f$ is irreducible then $N(f)$ is irreducible in $F[x]$.
 \end{theorem}

 \begin{proof}
  $N(f)$ is a two-sided multiple of $f$ in $R$, therefore the bound $f^*$ of $f$ must divide $N(f)$ in $R$. Since  $(f,t)_r=1,$ we know that $f^*\in C(R)$ and therefore $f^*$ equals $h$ up to some invertible factor in $F$. Thus
  $h(t)=\hat{h}(u^{-1}t^n)$ must divide $N(f)$ in $R$. Write $N(f)=g (t)h(t)$ for some $g\in R$.
 By Theorem \ref{thm:norm3} we have ${\rm deg}(N(f))=dmn$ in $R$.
Comparing degrees in $R$ we obtain ${\rm deg} N(f)={\rm deg}(g(t))+dmn=dmn$, which implies ${\rm deg}(g)=0$, i.e. $g(t)=a\in A^\times$.
  This gives $N(f)=ah(t)=a \hat{h}(u^{-1}t^n)$.
  \\ (i) If $\hat{h}$ is irreducible in $F[x]$ then $N(f)$ is irreducible in $F[x]$ (Lemma \ref{le:easy2}), thus $f$ is irreducible in $R$.
 \\ (ii) If $f$ is irreducible then $\hat{h}$ is irreducible in $F[x]$ (Lemma \ref{mclm exists}), and so again $N(f)$ is irreducible in $F[x]$.
  \end{proof}

 We can therefore again determine the similarity classes of irreducible polynomials appearing in a factorization of certain $f$, and show that any order is possible for the appearance of these similarity classes:

\begin{corollary}\label{cor:reducibleII}
 Let $A=(E/C,\gamma,a)$ be a cyclic algebra over $C$ of degree $d$ such that $\sigma|_E\in {\rm Aut}(E)$, $u\in E$, and  $\gamma\circ\sigma=\sigma\circ\gamma$.  Let $f\in A[t;\sigma]$
  be a monic  polynomial, such that $(f,t)_r=1$ and that ${\rm deg}(\hat{h})=dm$.
\\ (i) If $N(f(t))=\hat{h}_1\cdots \hat{h}_l$ such that $\hat{h}_i\in F[x]$ is irreducible,  $1\leq i\leq l$, then there exists a irreducible decomposition $f=f_1\cdots f_l$
of $f$ into irreducible factors $f_i$, such that  $N(f_{i})=\hat{h}_i$ for all $i$, $1\leq i\leq l$.
Moreover,  ${\rm deg}(f_i)={\rm deg}(N(f_{i}))/d$.
\\ (ii) Assume that $N(f)$ is the product of $l$ distinct irreducible factors $\hat{h}_1\cdots \hat{h}_l$ in $ F[x]$. Then $f$ has exactly $l!$ factorizations into irreducible factors  corresponding to each possible ordering of the factors of $N(f)$.
\end{corollary}

The proof is analogous to the one of Corollary \ref{cor:rough factorization}.

 Corollaries \ref{cor:rough factorization} (i) and \ref{cor:reducibleII}
(i) generalize \cite[Lemma 2.1.18]{CaB} which uses a different method of proof.

\begin{theorem}\label{cor:norm2}
Let $A=(E/C,\gamma, a)$ be a cyclic algebra over $C$ of degree $d$, and
 $f(t)=a_0+a_1t+\dots+a_mt^m \in C[t;\sigma]\subset  R$ be a polynomial of degree $m$,   such that $(f,t)_r=1$ and  ${\rm deg}(\hat{h})=dm$.
 Then $f$ is reducible in $R$ and has at least $d$ irreducible factors.
\end{theorem}

\begin{proof}
By Proposition \ref{prop:norm all in center},
 $$N(f(t))=(N_{C/F}(a_0)+\dots + (-1)^{m(n-1)}N_{C/F}(a_m)x^m)^d$$
  is clearly reducible in $F[x]$.  Since  ${\rm deg}(\hat{h})=dm$,   $f$ is reducible in $R$. The at least $d$ irreducible divisors $\hat{g_1},\dots,\hat{g_l}$ of
 $N(f(t))$ correspond to a decomposition $f=f_1\cdots f_l$
of $f$ into irreducible factors, such that  $N(f_i)=\hat{g_j}$ for some suitable $j$, for all $i$, $1\leq i\leq l$. The degree of $f_i$ in $R$ equals the degree of $\hat{g_j}$ in $F[x]$ for a suitable $j$.
\end{proof}

 Note that if $\sigma$ is an inner automorphism  (i.e. $n=1$ and so $R=D[t]$ by a change of indeterminants), $d$ is prime and $f$ not right-invariant, then any $f$ of degree $m$ will satisfy  ${\rm deg}(\hat{h})=dm$.

More generally, we also have by \cite[Proposition, p. 304]{Pierce}:

\begin{proposition}
Let $A=(E/C, \gamma, a)$ be a cyclic algebra.  Let $B$ be a central simple algebra over $K(x)$ that is a subalgebra of  $(A(x), \sigma, ut )$ and assume that $K(x)/F(x)$
is a finite field extension. Then every $f(t)\in A[t;\sigma]\cap B$ with ${\rm deg}(h)=dmn$  is reducible.
\end{proposition}

\begin{proof}
Let $K(x)/F(x)$ be a finite field extension of degree $c$ and $B$ be a central simple algebra over $K(x)$ of degree $b$.
For all $f(t)\in B\cap A[t;\sigma]$  we have
$N(f(t))=N_{K(x)/F(x)}(N_{B/K(x)}(f(t)))^e$ for a suitable integer $e$ with $d=bce$ \cite[Proposition, p. 304]{Pierce}.
Thus $N(f)$ is reducible and hence so is $f$, since we assumed that ${\rm deg}(h)=dmn$.
\end{proof}

%
%

\section{The norm condition for differential polynomials in characteristic $p$}\label{sec:norm}

\subsection{The  case that $R=K[t;\delta]$ }
We use the notation from \ref{sec_norm3}:
let $F$ be a field of characteristic $p$ and $K/F$ be a field extension which is purely inseparable of exponent one. Let $\delta$ be a derivation on $K$ that is algebraic with ${\rm Const}(\delta)=F$, and $g(t)=t^{p^e}+c_1t^{p^{e-1}}+\dots+ c_et \in F[t]$
 the minimum polynomial of $\delta$.
Let $R=K[t;\delta]$ and $C(R)=F[x]$, where $x=g(t)$.
Let $N$ be the norm of $ (K(x), \widetilde{\delta}, x )$.

\begin{theorem} \label{prop:norm3}
(i) For all $f\in R$ we have $N(f)\in F[x]$ and $f$ divides $N(f)$.
\\ (ii) If $f(t)=a_0+a_1t+\dots+a_mt^m\in R=K[t;\delta]$ has degree $m$, then
 $$N(f(t))=(-1)^{m(p^e-1)} a_m ^{p^e}x^m+\dots.$$
\end{theorem}

Part (i) of this proof follows by an analogous argument to Theorem \ref{le:Jacobsongeneralized}.
More precisely, actually $N(f)=f(t)^\sharp f(t)=f(t) f(t)^\sharp$ and $f(t)^\sharp\in R$, cf. \cite[(1.6.12) and p.~31]{J96}.

Part (ii) follows similarly to Theorem \ref{prop:correctedII}.
This generalizes and refines \cite[p. 31]{J96}, which states that  for $f(t)=a_0+a_1t+\dots+a_mt^m\in R$ we have
$N(f(t))=\pm a_m^{p^e}x^m + \dots $, if $m<p^e$, where the omitted terms are of lower degree than $m$ in $x$ (no proof was given).

To find the constant term of $N(f)$ in Theorem \ref{prop:norm3} is difficult, let us look at one example:

\begin{example}
 Let $p^e=5$, $f(t)=t^4+a$ for some $a\in K^{\times}$, and $g(t)=t^5+t$. Computing $\rho(f(t))$ yields \begin{equation*}
\left(\begin{array}{ccccc}
a &0 &0 &0 &1\\
\delta(a)+x &a-1 &0 &0 &0\\
\delta^2(a) &2\delta(a)+x &a-1 &0 &0\\
\delta^3(a) &3\delta^2(a) &3\delta(a)+x &a-1 &0\\
\delta^4(a) &4\delta^3(a) &6\delta^2(a) &4\delta(a)+x &a-1
\end{array}
\right).
\end{equation*}
Setting $x=0$ and taking the determinant of $\rho(f(t))$ shows that the constant term of $N(f(t))$ is
\begin{align*}
&a^5-4a^4+a^3[6+\delta^4(a)]-a^2[4+3\delta^4(a)+8\delta(a)\delta^3(a)+6\delta^2(a)^2]\\
&+a[1+3\delta^4(a)+12\delta^2(a)^216\delta(a)\delta^3(a)+36\delta(a)^2\delta^2(a)]\\
&-[\delta^4(a_0)+8\delta(a)\delta^3(a)+6\delta^2(a)^2+36\delta(a)^2\delta^2(a)+24\delta(a)^4].
\end{align*}
\end{example}

It is possible to compute special cases though:

\begin{proposition}  For $f(t)=g(t)+a$ for some $a\in K$,
$$N(f(t))=(x+a)^{p^e}.$$
\end{proposition}
\begin{proof}
Following the proof of Theorem \ref{prop:norm3}, we substitute $x=g(t)$ so $f(t)=x+a\in K[x][t;\delta]$. Computing the left regular representation $\rho:K[t;\sigma]\to M_{p^e}(K[x])$, it follows that $\rho(f(t))$ is a lower triangular matrix where each diagonal entry is equal to $x+a$. As the determinant of a triangular matrix is the product of its diagonal entries, the result follows.
\end{proof}

 Since $N(f)\in F[x]=C(R)$,  $f^*$  divides $N(f)$ in $F[x].$  Thus
${\rm deg}(f^*)\leq p^e \cdot {\rm deg}(f),$
 because the degree of $N(f)$ as a polynomial in $R$ is $m\cdot p^e$ by Theorem \ref{prop:norm3} (ii).
In particular, if ${\rm deg}(f^*)=m$ in $F[x]$, comparing degrees in $F[x]$ shows that $N(f)$ is the bound of $f$.

The bound
 $f^*$ is up to some scalar $\alpha\in K^\times$ equal to the minimal central left multiple $h$ of $f$, where as before we write $h(t)=\hat{h}(g(t))$ with $\hat{h}(x)\in F[x]$ monic.

\begin{lemma}\label{le:easy3}
Let $f\in R$. If $N(f)$ is irreducible in $F[x],$ then $f$ is irreducible in $R$.
\end{lemma}

The proof is identical to the one of Lemma \ref{le:easy2}.

 \begin{theorem} \label{thm:norm5}
 Let $f\in R$ be a polynomial of degree $m$.  Suppose that ${\rm deg}(\hat{h})=m$.
 \\ (i) If $\hat{h}$ is irreducible in $F[x]$ then $f$ is irreducible in $R$.
  \\ (ii) If $f$ is irreducible then $N(f)$ is irreducible in $F[x]$.
 \end{theorem}

The proof is identical to the one of Theorem \ref{thm:norm4}, using Theorem \ref{prop:norm3}, Proposition \ref{mclm exists2} and  Lemma \ref{le:easy3}.

\begin{corollary}\label{cor:rough factorization2}
Let $f\in R$ be a monic  polynomial of degree $m$. Suppose that ${\rm deg}(\hat{h})=m$.
\\ (i) If $N(f(t))=\hat{h}_1\cdots \hat{h}_l$ such that $\hat{h}_i\in F[x]$ is irreducible,  $1\leq i\leq l$, then there exists a decomposition $f=f_1\cdots f_l$
of $f$ into irreducible factors, such that  $N(f_{i})=\hat{h}_i$ for all $i$, $1\leq i\leq l$. The degree of $f_{i}$ in $R$ equals the degree of $\hat{h}_i$ in $F[x]$ for all $i$.
\\ (ii) Assume that $N(f)$ is the product of $l$ distinct irreducible factors $\hat{h}_1\cdots \hat{h}_l$ in $ F[x]$. Then $f$ has exactly $l!$ irreducible decompositions corresponding to each possible ordering of the factors of $N(f)$.
\end{corollary}

The proof is identical to the one of Corollary \ref{cor:rough factorization}, using Theorem \ref{prop:norm3} and  Lemma \ref{le:easy3}.

\subsection{The case that $R=D[t;\delta]$}

Let $C$ be a field  of characteristic $p$ and $D$  a central division algebra over $C$
 of degree $d$.
Let  $\delta$ be a derivation of $D$, such that $\delta|_C$ is algebraic with minimum polynomial
$g(t)=t^{p^e}+c_1t^{p^{e-1}}+\dots+ c_et\in F[t]$
 of degree $p^e$, where $F={\rm Const}(\delta)\cap C$. Write $g(t)=t^{p^e}+g_0(t)$.
 Then $g(\delta)=id_{d_0}$ is an inner
derivation of $D$ and w.l.o.g. we choose $d_0\in {\rm Const}(\delta)$  so that $\delta(d_0)=0$.
$R=D[t;\delta]$ has center $C(R)=F[x]=F[(t^{p^e}+g_0(t)+d_0)]$.

 Let $N$ be the reduced norm of $D(t; \delta ) \cong (D(x), \widetilde{\delta}, d_0+x )$ of degree
 $p^{e}d$ \cite[p. 23]{J96}.

  \begin{proposition}\label{prop:norm5}
   Let $f(t)=a_0+a_1t+\dots+a_mt^m\in C[t;\delta]\subset R$, then
   $$N(f(t))=( (-1)^{m(p^e-1)} a_m ^{p^e}x^m + \cdots  )^d.$$
  \end{proposition}

\begin{proof}
$\widetilde{C}=(C(x),\widetilde{\delta}|_{C(x)}, d_0+x )=(C(x), \widetilde{\delta},x)$ is a subalgebra of degree $p^e$ of $(D(x), \widetilde{\delta}, d_0+x )$ over $F(x)$. Note that $\widetilde{\delta}|_{C(x)}:C(x)\longrightarrow C(x)$ is well-defined: since $\delta|_{C}$ is an algebraic derivation on $C$ by our assumption,
 we can canonically extend it to get the algebraic derivation $\widetilde{\delta}|_{C(x)}:C(x)\longrightarrow C(x)$. Thus
$$N(f(t))=(N_{\widetilde{C}/F(x)}(f(t)))^d$$
for all $f(t)\in C[t;\sigma] $ by \cite[Proposition. p.~304]{Pierce}. This yields the assertion.
\end{proof}

\begin{theorem}
 Let $D$ have a subfield $E$ of degree $d$.
 \\ (i) For all $f\in R$ we have $N(f)\in F[x]$.
 \\ (ii) $f$ divides $N(f)$.
\end{theorem}

\begin{proof}
(i) The set $\{ 1,t , \dots, t^{p^e-1}\}$ generates $D[t; \delta]$ over $D[x]$. The algebra
  $(D(x), \widetilde{\delta}, d_0+x )$ is central simple of degree $p^ed$ over $F(x)$ with subalgebra $D(x)$. We regard $(D(x), \widetilde{\delta}, d_0+x )$ as a left module over $D(x)$.
Since $C(D[t; \delta]) = F[x] \subset D[x]$, the set $\{ 1,t , \dots, t^{n-1}\}$ also generates $D(t; \delta)$ as a left $D(x)$-module, and hence is a basis.
 Since
  $$D[t;\delta] =\bigoplus_{i=0}^{ p^e-1} D[x]t^i,$$
  and by substituting $x=g(t)=t^{p^e}+g_0(t)$,
     every $f\in R \subset (D(x), \widetilde{\delta}, d_0+x )$ can be written as a polynomial $1,t, \dots, t^{p^e-1}$ with coefficients in $D[x]$.
We therefore obtain a representation  $\rho$ of $(D(x), \widetilde{\delta}, d_0+x )$ by matrices in $M_{p^e}(D(x))$ by writing
$$t^i f(t)=\sum_{j=0}^{n-1} \rho_{ij}(f(t))t^j$$
for all $f\in R\subset (A(x), \widetilde{\delta}, d_0+x)$ and $0\leq i\leq n-1$,
where $\rho_{ij}(f(t))$ is the $(i,j)^\text{th}$ entry of $\rho(f(t))$. The matrix $\rho(f(t))$ then has entries in $D[x]$ for every $f\in R$.

  Let $\{v_1,\dots,v_d\}$ be a canonical basis of $D$ as a left $E$-module.
  Then the set $\{v_1,\dots,v_d,\\ v_1t,\dots,v_dt,\dots v_dt^{p^e-1}\}$ is a basis of $(D(x), \widetilde{\delta}, d_0+x )$ as a left module over $E(x)$. We can therefore analogously obtain a representation $\rho$ of $(D(x), \widetilde{\delta}, d_0+x )$ by matrices in $M_{dp^e}(E(x))$ with respect to that basis.
   For $f(t)\in R$, the matrix $\rho(f(t))$ has entries in $E[x]$, therefore it follows that
$N(f(t))={\rm det} (\rho(f(t)))\in E[x]\cap F(x) =F[x].$
 \\ (ii)   Similarly as in (i) (and similar to \cite[Proposition, p. 295]{Pierce}), it can be shown that all the coefficients of the characteristic polynomial of $\rho(f(t))$ are contained in
  $F[x]$  and thus $f(t)^\sharp \in R$ by \cite[(1.6.12)]{J96}. Since $N(f(t))=f(t)f(t)^\sharp=f(t)^\sharp f(t)$, it follows that $f(t)$ divides $N(f)$ in $R$.
\end{proof}

\begin{theorem}\label{thm:deg(N)=dm for delta}
Let $D$ have a subfield $E$ of degree $d$ and let $\omega:D\to M_d(E)$ be the
 left regular representation of $D$. Then for any $f\in R$ of degree $m$,
$$N(f)=\pm{\rm det}(\omega(a_m))^{p^e}x^{dm}+\dots$$
In particular, $N(f)$ has degree $dm$.
\end{theorem}

The proof follows analogously to the proof of Theorem \ref{thm:norm6}.

\begin{corollary}
Let $D=(E,\delta_0,a)$ be a differential algebra, $\delta|_E$ be a derivation on $E$, and let $f\in D[t;\delta]$ have coefficients in $E$ and be monic. Then $N(f(t))=\pm x^{dm}+\dots$.
\end{corollary}
\begin{proof}
Through direct computations of the left regular representation of $D$, we see that for each $a\in E$, $\omega(a)$ is a lower triangular matrix with each entry on the lead diagonal equal to $a$. Hence the result follows analogously to Theorem \ref{thm:deg(N)=dm for delta}.
\end{proof}

 Since $N(f)\in F[x]=C(R)$, the bound $f^*$ of $f$ divides $N(f)$ in $F[x],$ so that if
 ${\rm deg}(f^*)=dm$ in $F[x],$ it follows again immediately that $N(f)$ is the bound of $f$.

\begin{lemma}\label{le:easy4}
Let $f\in R$. If $N(f)$ is irreducible in $F[x],$ then $f$ is irreducible in $R$.
\end{lemma}

 \begin{theorem} \label{thm:norm5}
 Let $f\in R$ be a polynomial of degree $m$.  Suppose that ${\rm deg}(\hat{h})=dm$.
 \\ (i) If $\hat{h}$ is irreducible in $F[x]$ then $f$ is irreducible in $R$.
  \\ (ii) If $f$ is irreducible then $N(f)$ is irreducible in $F[x]$.
  \end{theorem}

\begin{corollary}\label{cor:rough factorization3}
Let $f\in R$ be a monic  polynomial of degree $m$. Suppose that ${\rm deg}(\hat{h})=dm$.
\\ (i) If $N(f(t))=\hat{h}_1\cdots \hat{h}_l$ such that $\hat{h}_i\in F[x]$ is irreducible,  $1\leq i\leq l$, then there exists a decomposition $f=f_1\cdots f_l$
of $f$ into irreducible factors, such that  $N(f_{i})=\hat{h}_i$ for all $i$, $1\leq i\leq l$. The degree of $f_{i}$ in $R$ equals the degree of $\hat{h}_i$ in $F[x]$ for all $i$.
\\ (ii) Assume that $N(f)$ is the product of $l$ distinct irreducible factors $\hat{h}_1\cdots \hat{h}_l$ in $ F[x]$. Then $f$ has exactly $l!$ irreducible decompositions corresponding to each possible ordering of the factors of $N(f)$.
\end{corollary}

The proofs of the above results are identical to the ones of Lemmata \ref{le:easy2}, \ref{le:easy3}, Theorem \ref{thm:norm4} and  Corollary \ref{cor:rough factorization2}.


\end{document}